\documentclass[final]{amsart}
\usepackage{amssymb}
\usepackage{enumerate}

\makeatletter%
\def\th@mytheorem{%
  \let\thm@indent\noindent
  \thm@headfont{\bfseries}
    \itshape
}
\def\th@myremark{%
  \let\thm@indent\noindent
  \thm@headfont{\bfseries}
}
\makeatother%
\theoremstyle{mytheorem}
\newtheorem{Theorem}{Theorem}[section]
\newtheorem{Lemma}[Theorem]{Lemma}
\theoremstyle{myremark}
\newtheorem{Remark}[Theorem]{Remark}
\def\bs{\boldsymbol}
\def\bu{$\bullet$\quad}
\def\BT{\bs{T}}
\def\BU{\bs{U}}
\def\BX{\bs{X}}
\def\BY{\bs{Y}}
\def\Cal{\mathcal}
\def\Case#1{{\rm(#1)}}
\def\CO{\mathcal O}
\def\CC{\mathbb C}
\def\CF{\mathcal F}
\def\halfskip{\vskip 6pt plus 1pt minus 1pt}
\def\NN{\mathbb N}
\def\Omega{\varOmega}
\let\phi=\varphi
\def\PSH{\mathcal{PSH}}
\def\PP{\mathbb P}
\def\pr{\operatorname{pr}}
\def\ttimes{\times\dots\times}
\def\too{\longrightarrow}
\def\tuu{\longmapsto}
\def\wdht{\widehat}
\def\wdtl{\widetilde}

\begin{document}
\title
{Cross theorem with singularities \\  pluripolar vs.~analytic case}

\author{Marek Jarnicki}
\address{Jagiellonian University, Institute of Mathematics,
{\L}ojasiewicza 6, 30-348 Krak\'ow, Poland} \email{Marek.Jarnicki@im.uj.edu.pl}

\author{Peter Pflug}
\address{Carl von Ossietzky Universit\"at Oldenburg, Institut f\"ur Mathematik,
Postfach 2503, D-26111 Oldenburg, Germany}
\email{pflug@mathematik.uni-oldenburg.de}
\thanks{The research was partially supported by
the grant no.~N N201 361436 of the Polish Ministry of Science and Higher Education
and the DFG-grant 436POL113/103/0-2.}

\subjclass{32D15}

\keywords{separately holomorphic function, cross theorem with singularities}

\begin{abstract}
We prove that in the extension theorem for separately holomorphic functions on an $N$--fold
cross with singularities the case of analytic singularities follows from the case of pluripolar
singularities.
\end{abstract}

\maketitle


\section{Introduction. Main result}
Throughout the paper we will work in the following geometric context ---
details may be found in \cite{JarPfl2007}, see also \cite{JarPfl2003a}, \cite{JarPfl2003b}.

We fix an integer $N\geq2$ and let $D_j$ be a (connected) \textit{Riemann domain of holomorphy}
over $\CC^{n_j}$, $j=1,\dots,N$. Let $\varnothing\neq A_j\subset D_j$ be \textit{locally pluriregular},
$j=1,\dots,N$.

We will use the following conventions. For arbitrary $B_j\subset D_j$, $j=1,\dots,N$, we write
$B'_j:=B_1\ttimes B_{j-1}$, $j=2,\dots,N$, $B''_j:=B_{j+1}\ttimes B_N$, $j=1,\dots,N-1$.
Thus, for each $j\in\{1,\dots,N\}$, we may write $B_1\ttimes B_N=B'_j\times B_j\times B''_j$
(with natural exceptions for $j\in\{1,N\}$). Analogously, a point
$a=(a_1,\dots,a_N)\in D_1\ttimes D_N$ will be frequently written as $a=(a'_j,a_j,a''_j)$, where
$a'_j:=(a_1,\dots,a_{j-1})$, $a''_j:=(a_{j+1},\dots,a_N)$ (with obvious exceptions for
$j\in\{1,N\}$).

We define an \textit{$N$--fold cross}
$$
\BX=\BX(D_1,\dots,D_N; A_1,\dots,A_N)=\BX((D_j,A_j)_{j=1}^N):=\bigcup_{j=1}^N A'_j\times D_j\times A''_j.
$$
One may prove that $\BX$ is connected.

More generally, for arbitrary \textit{pluripolar} sets $\Sigma_j\subset A'_j\times A''_j$,
$j=1,\dots,N$, we define an \textit{$N$--fold generalized cross}
\begin{multline*}
\BT=\BT(D_1,\dots,D_N; A_1,\dots,A_N; \Sigma_1,\dots,\Sigma_N)=\BT((D_j,A_j,\Sigma_j)_{j=1}^N):\\
=\bigcup_{j=1}^N
\Big\{(a'_j,z_j,a''_j)\in A'_j\times D_j\times A''_j: (a'_j,a''_j)\notin\Sigma_j\Big\}\subset\BX.
\end{multline*}
We say that \textit{$\BT$ is generated by $\Sigma_1,\dots,\Sigma_N$}.
Obviously, $\BX=\BT((D_j,A_j,\varnothing)_{j=1}^\infty)$.

Observe that any $2$--fold
generalized cross is in fact a $2$--fold cross, namely
\begin{multline*}
\BT(D_1,D_2;A_1,A_2;\Sigma_1,\Sigma_2)=
(D_1\times(A_2\setminus\Sigma_1))\cup((A_1\setminus\Sigma_2)\times D_2)\\
=\BX(D_1,D_2;A_1\setminus\Sigma_2,A_2\setminus\Sigma_1).
\end{multline*}
Notice that for $N\geq3$ the geometric structure of $\BT$ is essentially different.

\halfskip

Let $h_{A_j,D_j}$ denote the relative extremal function of $A_j$ in $D_j$, $j=1,\dots,N$.
Recall that
$$
h_{A,D}:=\sup\{u\in\PSH(D): u\leq1,\;u|_A\leq0\}.
$$
Put
$$
\wdht{\BX}:=\{(z_1,\dots,z_N)\in D_1\ttimes D_N: h^\ast_{A_1,D_1}(z_1)+\dots+h^\ast_{A_N,D_N}(z_N)<1\},
$$
where ${}^\ast$ stands for the upper semicontinuous regularization.
One may prove that $\wdht{\BX}$ is a (connected) domain of holomorphy and $\BX\subset\wdht{\BX}$.

\halfskip

Let $M\subset\BT$ be \textit{relatively closed}.
We say that a function $f:\BT\setminus M\too\CC$ is \textit{separately holomorphic on $\BT\setminus M$}
(we write $f\in\CO_s(\BT\setminus M)$) if for any
$j\in\{1,\dots,N\}$ and $(a'_j,a''_j)\in (A'_j\times A''_j)\setminus\Sigma_j$,
the function
$
D_j\setminus M_{(a'_j,\cdot,a''_j)}\ni z_j\tuu f(a'_j,z_j,a''_j)\in\CC
$
is holomorphic in $D_j\setminus M_{(a'_j,\cdot,a''_j)}$, where
$
M_{(a'_j,\cdot,a''_j)}:=\{z_j\in D_j: (a'_j,z_j,a''_j)\in M\}
$
is the \textit{fiber of $M$ over $(a'_j,a''_j)$}.

We are going to discuss the following extension theorem with singularities proved in
\cite{JarPfl2003a}, \cite{JarPfl2003b}, see also \cite{JarPfl2007}.

\begin{Theorem}[Extension theorem with singularities for crosses]\label{ThmExtX}
Under the above assumptions, let $\BT\subset\BX$ be an $N$--fold generalized cross
and let $M\subset\BX$ be a relatively closed set such that

\Case{\dag}\quad for all $j\in\{1,\dots,N\}$ and
$(a'_j,a''_j)\in(A'_j\times A''_j)\setminus\Sigma_j$, the fiber $M_{(a'_j,\cdot,a''_j)}$ is pluripolar.

Then there exist an $N$--fold generalized cross $\BT'\subset\BT$
(generated by pluripolar sets $\Sigma'_j\subset A'_j\times A''_j$ with
$\Sigma'_j\supset\Sigma_j$, $j=1,\dots,N$)
and a relatively closed pluripolar set $\wdht M\subset\wdht{\BX}$ such that:
\begin{enumerate}[{\rm(A)}]
\item\label{A} $\wdht M\cap\BT'\subset M$,

\item\label{B} for every $f\in\CO_s(\BX\setminus M)$ the exists an $\wdht f\in\CO(\wdht{\BX}\setminus\wdht M)$
such that $\wdht f=f$ on $\BT'\setminus M$,

\item\label{C} the set $\wdht M$ is minimal in that sense that each point of $\wdht M$
is singular with respect to the family $\wdht\CF:=\{\wdht f: f\in\CO_s(\BX\setminus M)\}$
--- cf.~\cite{JarPfl2000}, \S\;3.4,

\item\label{D} if for any $j\in\{1,\dots,N\}$ and
$(a'_j,a''_j)\in(A'_j\times A''_j)\setminus\Sigma_j$, the fiber is thin, then $\wdht M$ is analytic
in $\wdht{\BX}$ (and in view of {\rm(\ref{C})}, either $\wdht M=\varnothing$ or
$\wdht M$ must be of pure codimension one --- cf.~\cite{JarPfl2000}, \S\;3.4),

\item\label{E} if $M=S\cap\BX$, where $S\varsubsetneq U$ is an analytic subset of an open
connected neighborhood $U\subset\wdht{\BX}$ of $\BX$, then $\wdht M\cap U_0\subset S$ for an
open neighborhood $U_0\subset U$ of $\BX$ and $\wdht f=f$ on $\BX\setminus M$ for every
$f\in\CO_s(\BX\setminus M)$,

\item\label{F} in the situation of \Case{E}, if $U=\wdht{\BX}$, then $\wdht M$ is the union
of all one codimensional irreducible components of $S$.
\end{enumerate}
\end{Theorem}

Observe that in the situation of \Case{E}, if $M=S\cap\BX$ and \Case{\dag} is satisfied, then
for any $j\in\{1,\dots,N\}$ and $(a'_j,a''_j)\in(A'_j\times A''_j)\setminus\Sigma_j$, the fiber
$M_{(a'_j,\cdot,a''_j)}$ is analytic (in particular, thin) and
therefore, by \Case{D}, the set $\wdht M$ must be analytic.

It has been conjectured (in particular, in  \cite{JarPfl2003b}) that in fact conditions
\Case{E--F} are consequences of \Case{A--D}. Notice that the method of proof of \Case{E--F}
used in \cite{JarPfl2003a} is essentially different than the one of \Case{A--D} in
\cite{JarPfl2003b}. The aim of this paper is to prove this conjecture which finally leads to
a uniform presentation of the cross theorem with singularities. Our main result is the
following theorem.

\begin{Theorem}\label{ThmMain}
Properties \Case{E--F} follow from \Case{A--D}.
\end{Theorem}

\section{Proof of Theorem \ref{ThmMain}}
Roughly speaking, the main idea of the proof is to show that if $\wdht M\cap\BT'\subset M$,
then $\varnothing\neq\wdht M\cap\Omega\subset S$ for an open set $\Omega\subset\wdht{\BX}$.
We will need the following extension theorems (without singularities).

\begin{Theorem}\label{ThmExtT}
{\rm (a) (Classical cross theorem --- cf.~e.g.~\cite{AleZer2001}.)}
Under the above assumptions, every function $f\in\CO_s(\BX)$ extends
holomorphically to $\wdht{\BX}$.

{\rm (b) (Cross theorem for generalized crosses --- cf.~\cite{JarPfl2003b}, \cite{JarPfl2007}.)}
Under the above assumptions,
every function $f\in\CO_s(\BT)\cap\Cal C(\BT)$ extends holomorphically to $\wdht{\BX}$.
\end{Theorem}

\begin{Remark}
(a) The assumptions in Theorem \ref{ThmExtT}(b) may be essentially weakened. Namely,
using the same method of proof as in \cite{JarPfl2003b}, one may easily
show that every function $f\in\CO_s(\BT)$ such that
for any $j\in\{1,\dots,N\}$ and $b_j\in D_j$, the function
$A'_j\times A''_j\setminus\Sigma_j\ni (z'_j,z''_j)\tuu f(z'_j,b_j,z''_j)$
is continuous, extends holomorphically to $\wdht{\BX}$.

(b) We point out that it is still an \textit{open problem} whether for $N\geq3$ and arbitrary $\BT$,
Theorem \ref{ThmExtT}(b) remains true for every $f\in\CO_s(\BT)$.
\end{Remark}

\begin{Remark}\label{RemProp}
If for all $j\in\{1,\dots,N\}$ and
$(a'_j,a''_j)\in(A'_j\times A''_j)\setminus\Sigma_j$, the fiber $M_{(a'_j,\cdot,a''_j)}$
is pluripolar, then the sets
$$
\{(a'_j,a_j,a''_j)\in A'_j\times A_j\times A''_j: (a'_j,a''_j)\notin\Sigma_j,\;
a_j\notin M_{(a'_j,\cdot,a''_j)}\},\quad j=1,\dots,N,
$$
are non-pluripolar (cf.~\cite{JarPfl2007}).
\end{Remark}

\begin{Lemma}\label{LemQT}
Let $Q\subset\wdht{\BX}$ be an arbitrary analytic set of pure codimension one and let $\BT\subset\BX$
be an arbitrary generalized cross. Then $Q\cap\BT\neq\varnothing$.
\end{Lemma}

\begin{proof}
Suppose that $Q\cap\BT=\varnothing$. Since $Q$ is of pure codimension one,
$\wdht{\BX}\setminus Q$ is a domain of holomorphy, and therefore, there exists
a $g\in\CO(\wdht{\BX}\setminus Q)$ such that $\wdht{\BX}\setminus Q$ is the domain of
existence of $g$. Since
$\BT\subset\wdht{\BX}\setminus Q$, we conclude that $f:=g|_{\BT}\in\CO_s(\BT)\cap\Cal C(\BT)$.
By Theorem \ref{ThmExtT} there exists an $\wdht f\in\CO(\wdht{\BX})$ such that
$\wdht f=f$ on $\BT$. Consequently, since $\BT$ is non-pluripolar,
we conclude that $\wdht f=g$ on $\wdht{\BX}\setminus Q$.
Thus $g$ extends holomorphically to $\wdht{\BX}$; a contradiction.
\end{proof}

\begin{Lemma}\label{LemEF}
Condition \Case{F} follows from \Case{A--E}.

Thus to prove Theorem \ref{ThmMain} we only need to check that \Case{E} follows
from \Case{A--D}.
\end{Lemma}

\begin{proof}
Indeed, let $S\varsubsetneq\wdht{\BX}$ be an analytic set, $M:=S\cap\BX$, and assume that
\Case{A--E} hold true.
Let $S_0$ be the union of all irreducible components of $S$ of codimension one.
Consider two cases:

$S_0\neq\varnothing$: Similarly as in the proof of Lemma \ref{LemQT},
there exists a non-continuable function $g\in\CO(\wdht{\BX}\setminus S_0)$. Then
$f:=g|_{\BX\setminus M}\in\CO_s(\BX\setminus M)$ and, therefore (by \Case{E}),
there exists an $\wdht f\in\CO(\wdht{\BX}\setminus\wdht M)$ with $\wdht f=f$ on $\BX\setminus M$.
Observe that (by \Case{E})
$\BX\setminus M\subset(\wdht{\BX}\setminus\wdht M)\cap(\wdht{\BX}\setminus S)\subset
\wdht{\BX}\setminus(S_0\cup\wdht M)$. The set $\BX\setminus M$ is non-pluripolar
(Remark \ref{RemProp}). Hence
$\wdht f=g$ on $\wdht{\BX}\setminus(S_0\cup\wdht M)$. Since $g$ is non-continuable,
we conclude that $S_0\subset\wdht M$.

The set $\wdht M$, as a non-empty singular set,
must be of pure codimension one. Since $\wdht M\cap U_0\subset S$ and
$Q\cap U_0\neq\varnothing$ for every irreducible component $Q$ of $\wdht M$
(by Lemma \ref{LemQT}), we conclude, using the identity principle for analytic sets, that
$\wdht M\subset S$ (cf.~\cite{Chi1989}, \S\;5.3). Consequently, $\wdht M\subset S_0$.

$S_0=\varnothing$: Suppose that $\wdht M\neq\varnothing$. Then $\wdht M$ must be of pure codimension one.
The above proof of the first part shows that $\wdht M\subset S$.
Since $S_0=\varnothing$, the codimension of $S$ is $\geq2$; a contradiction.
\end{proof}

\begin{Lemma}\label{LemMT}
Suppose that \Case{A--D} are true and in the situation of \Case{E}
we know that $\wdht M\cap\BX\subset M$. Then $\wdht f=f$ on $\BX\setminus M$.

Thus, the proof of \Case{E} reduces to the inclusion $\wdht M\cap U_0\subset S$.
\end{Lemma}

\begin{proof}
First observe that, in the situation of \Case{A--D}, if $\BT'\subset\BT''\subset\BX$,
where $\BT''$ is generated by pluripolar sets $\Sigma''_j\subset A'_j\times A''_j$ with
$\Sigma''_j\subset\Sigma'_j$, $j=1,\dots,N$, are such that:

\bu for all $j\in\{1,\dots,N\}$ and
$(a'_j,a''_j)\in(A'_j\times A''_j)\setminus\Sigma''_j$, the fiber $M_{(a'_j,\cdot,a''_j)}$
is pluripolar,

\bu $\wdht M\cap\BT''\subset M$,

\noindent then $\wdht f=f$ on $\BT''\setminus M$.

Indeed, fix a point $a\in\BT''\setminus M$. We may assume that
$$
a=(a'_N,a_N)\in(A'_N\setminus\Sigma''_N)\times(D_N\setminus M_{(a'_N,\cdot)}).
$$
Since $\wdht M_{(a'_N,\cdot)}\subset M_{(a'_N,\cdot)}$, the functions
$f(a'_N,\cdot)$ and $\wdht f(a'_N,\cdot)$ are holomorphic in the domain
$D_N\setminus M_{(a'_N,\cdot)}$. It suffices to show that they coincides on a non-pluripolar
subset of $D_N\setminus M_{(a'_N,\cdot)}$.

Take a $b_N\in A_N\setminus M_{(a'_N,\cdot)}$, put
$c=(c_1,\dots,c_N):=(a'_N,b_N)$ and let $r_0>0$ be so small that $\PP(c,r_0)\cap M=\varnothing$,
where $\PP(c,r_0)$ stands for the ``polydisc'' in sense of Riemann domains
(cf.~\cite{JarPfl2000}, \S\;1.1). Applying Theorem \ref{ThmExtT}(a) to the $N$--fold cross
$\BX_c:=\BX((\PP(c_j,r_0),A_j\cap\PP(c_j,r_0))_{j=1}^N)$ shows that there exist
$r\in (0,r_0)$ and $\wdtl f_c\in\CO(\PP(c,r))$       
such that $\wdtl f_c=f$ on $\PP(c,r)\cap\BX_c$. Since $\wdht f=f=\wdtl f_c$ on
the non-pluripolar set $\PP(c,r)\cap\BT'\setminus M$ (cf.~Remark \ref{RemProp})
and $\wdht M$ is singular (cf.~(\ref{D})),
we get $\PP(c,r)\cap\wdht M=\varnothing$ and $\wdht f=\wdtl f_c$ on $\PP(c,r)$.

Finally, $f(a'_N,\cdot)=\wdtl f_c(a'_N,\cdot)=\wdht f(a'_N,\cdot)$ on the non-pluripolar set
$\PP(b_N,r)\cap A_N$.

\halfskip

If $M$ is an analytic subset of $U$, then we may take
\begin{multline*}
\Sigma''_j:=\{(a'_j,a''_j)\in A'_j\times A''_j: M_{(a'_j,\cdot,a''_j)} \text{ is thin}\}\\
=\{(a'_j,a''_j)\in A'_j\times A''_j: M_{(a'_j,\cdot,a''_j)}\neq D_j\}.
\end{multline*}
Observe that $\BT'\subset\BT\subset\BT''$ and $\BT''\setminus M=\BX\setminus M$.
Thus, if we know that $\wdht M\cap\BX\subset M$, then
$\wdht f=f$ on $\BT''\setminus M=\BX\setminus M$.
\end{proof}

\begin{Lemma}
If condition \Case{E} is true with $U=\wdht{\BX}$ (and
arbitrary other elements), then it is true with general $U$.

Thus to prove Theorem \ref{ThmMain} we only need to check that \Case{E} with $U=\wdht{\BX}$
follows from \Case{A--D}.
\end{Lemma}

\begin{proof}
It suffices to show that for every $a\in\BX$ there exists an open neighborhood $U_a\subset U$
such that $\wdht M\cap U_a\subset S$.
We may assume that $a=(a_1,\dots,a_N)=(a'_N,a_N)\in A'_N\times D_N$. Let $G_N\Subset D_N$ be a domain
of holomorphy such that $G_N\cap A_N\neq\varnothing$, $a_N\in G_N$. Since
$\{a'_N\}\times G_N\subset\{a'_N\}\times D_N\subset\BX\subset U$, there exists an $r>0$
such that $\PP(a'_N,r)\times G_N\subset U$. Consider the $N$--fold cross
\begin{multline*}
\BY:=\BX(\PP(a_1,r),\dots,\PP(a_{N-1},r),G_N;\\ 
A_1\cap\PP(a_1,r),\dots,A_{N-1}\cap\PP(a_{N-1},r),A_N\cap G_N)\subset\BX.
\end{multline*}
Notice that $\wdht{\BY}\subset\PP(a'_N,r)\times G_N\subset U$. Consequently, the analytic set
$S_{\BY}:=S\cap\wdht{\BY}$ satisfies the special assumption ``$U=\wdht{\BX}$''
with respect to the cross $\BY$. Let $\wdht M_{\BY}$ be constructed according to \Case{A--D}
for $M_{\BY}:=S\cap\BY$. Using our assumption and Lemma \ref{LemEF}, we conclude that
$\wdht M_{\BY}\subset S_{\BY}$.

Since $a\in\wdht{\BY}$, it suffices to show that $\wdht M\cap\wdht{\BY}\subset\wdht M_{\BY}$.
Take an $f\in\CO_s(\BX\setminus M)$. Then
$f_{\BY}:=f|_{\BY\setminus M_{\BY}}\in\CO_s(\BY\setminus M_{\BY})$ and, therefore there exists
an $\wdht f_{\BY}\in\CO(\wdht{\BY}\setminus\wdht M_{\BY})$ with $\wdht f_{\BY}=f$ on
$\BY\setminus M_{\BY}$ (Lemma \ref{LemMT}). Since the set $\wdht M$ is singular, we must have
$\wdht M\cap\wdht{\BY}\subset\wdht M_{\BY}$.
\end{proof}

\begin{Lemma}\label{Lemh}
To prove \Case{E} with $U=\wdht{\BX}$ we may assume that $S=h^{-1}(0)$ with $h\in\CO(\wdht{\BX})$,
$h\not\equiv0$.
\end{Lemma}

\begin{proof} Since $\wdht{\BX}$ is pseudoconvex, $S$ may be written as
$
S=\{z\in\wdht\BX:  h_1(z)=\dots=h_k(z)=0\},
$
where $h_j\in\CO(\wdht{\BX})$, $h_j\not\equiv0$, $j=1,\dots,k$.
Put $S_j:=h_j^{-1}(0)$, $M_j:=S_j\cap\BX$, $j=1,\dots,k$.
Take an $f\in\CO_s(\BX\setminus M)$. Observe that
$f_j:=f|_{\BX\setminus M_j}\in\CO_s(\BX\setminus M_j)$. We have assumed that
for each $j$ there exists an $\wdht f_j\in\CO(\wdht\BX\setminus S_j)$
such that $\wdht f_j=f$ on $\BX\setminus M_j$.
Gluing the functions $(\wdht f_j)_{j=1}^k$
leads to an $\wdtl f\in\CO(\wdht\BX\setminus S)$ with $\wdtl f=\wdht f_j$
on $\wdht\BX\setminus S_j$, $j=1,\dots,k$. Therefore, $\wdtl f=f$ on $\BX\setminus S$.
Since $\wdht M$ is singular, we must have $\wdht M\subset S$.
\end{proof}

After all above preparations we are ready for the main part of the proof.

\begin{proof}
We may assume that $S=h^{-1}(0)$ with $h\in\CO(\wdht{\BX})$, $h\not\equiv0$.
Of course, we may assume that $\wdht M\neq\varnothing$.
Thus $\wdht M$ is of pure codimension one.
Recall that we only know that $\wdht M\cap\BT'\subset M$ and $\wdht f=f$ on $\BT'\setminus M$.
Let $\wdht M_0$ be an irreducible component of $\wdht M$.
By the identity principle for analytic sets we only need to show
that $\varnothing\neq\Omega\cap\wdht M_0\subset S$ for an open set $\Omega\subset\wdht{\BX}$.

For every point $a\in\wdht M_0$ there exist an $\rho_a>0$ and a defining function
$g_a\in\CO(\PP(a,\rho_a))$ for $\wdht M_0\cap\PP(a,\rho_a)$ (cf.~\cite{Chi1989}, \S\;2.9),
in particular, $\wdht M_0\cap\PP(a,\rho_a)=g_a^{-1}(0)$.
Using the Lindel\"of theorem, we may find a sequence
$(a_k)_{k=1}^\infty$ such that $\wdht M_0\subset\bigcup_{k=1}^\infty\PP(a_k,\rho_{a_k})$.

To get the main idea of the proof assume first that

(*)\quad there exist $k\in\NN$, $j\in\{1,\dots,N\}$, and
a point $b=(b'_j,b_j,b''_j)\in\wdht M_0\cap\PP(a_k,\rho_{a_k})$ such that
$(b'_j,b''_j)\in(A'_j\times A''_j)\setminus\Sigma'_j$ and
$g_{a_k}(b'_j,b_{j,1},\dots,b_{j,n_j-1},\cdot,b''_j)\not\equiv0$
in $\PP((a_k)_{j,n_j},\rho_{a_k})$, where $\PP((a_k)_j,\rho_{a_k})\ni
z_j=(z_{j,1},\dots,z_{j,n_j})$ (in local
coordinates); observe that $b\in\wdht M\cap\BT'\subset S$.

We may assume that $j=N$. Put $a:=a_k$, $\rho:=\rho_{a_k}$, $g:=g_{a_k}$, $n:=n_1+\dots+n_N$.
Let $b=(\wdtl b,b_n)\in\CC^{n-1}\times\CC$
in local coordinates in $\PP(a,\rho)$. Consequently, we may assume that
for certain $\wdtl r, r_n>0$ with $\PP(\wdtl b,\wdtl r)\times\PP(b_n,r_n)\subset\PP(a,\rho)$
we have:

\bu $g(\wdtl b,\cdot)$ has in the disc $\PP(b_n,r_n)$ the only zero at $z_n=b_n$ with multiplicity $p$,

\bu for every $\wdtl z\in\PP(\wdtl b,\wdtl r)$ the function $g(\wdtl z,\cdot)$
has in $\PP(b_n, r_n)$ exactly $p$ zeros counted with multiplicities.

In particular, the projection
$\wdht M_0\cap(\PP(\wdtl b,\wdtl r)\times\PP(b_n,r_n))\ni (z',z_n)\overset{\pi}\too z'\in\PP(\wdtl b,\wdtl r)$
is proper. It is known that there exists a relatively closed pluripolar set
$\Sigma\subset\PP(\wdtl b,\wdtl r)$ such that
$\pi|_{\pi^{-1}(\PP(\wdtl b,\wdtl r)\setminus\Sigma)}:
\pi^{-1}(\PP(\wdtl b,\wdtl r)\setminus\Sigma)\too
\PP(\wdtl b,\wdtl r)\setminus\Sigma$ is a holomorphic covering (cf.~\cite{Chi1989}, \S\;2.8).
Let $C:=((A'_N\setminus\Sigma_N)\cap\PP(b'_N,\wdtl r))\times
\PP((b_{N,1},\dots,b_{N,n_N-1}),\wdtl r)\subset\PP(\wdtl b,\wdtl r)$;
it is clear that $C$ is locally pluriregular.

Thus there exist a $\wdtl c\in C$,
$\overset{\approx}r>0$, and $\phi:\PP(\wdtl c,\overset{\approx}r)\too\PP(b_n,r_n)$ holomorphic
such that $\PP(\wdtl c,\overset{\approx}r)\subset\PP(\wdtl b,\wdtl r)$ and
the graph $\{(\wdtl z,\phi(\wdtl z)): \wdtl z\in\PP(\wdtl c,\overset{\approx}r)\}$ is an open
part of $\wdht M_0$. Thus $h(\wdtl z,\phi(\wdtl z))=0$,
$\wdtl z\in C\cap\PP(\wdtl c,\overset{\approx}r)$. Hence
$h(\wdtl z,\phi(\wdtl z))=0$,
$\wdtl z\in\PP(\wdtl c,\overset{\approx}r)$, which means that
$(\wdtl c,\phi(\wdtl c))\in\Omega\cap\wdht M_0\subset S$ for an open set $\Omega\subset\wdht{\BX}$.

\halfskip

We move to the general case. Let
$$
C_{j,k}=(\pr_{D'_j\times D''_j}(\PP(a_k,\rho_{a_k})\cap\wdht M_0))\cap((A'_j\times A''_j)\setminus\Sigma'_j),
\quad j=1,\dots,N,\;k\in\NN.
$$
Suppose that all the sets $C_{j,k}$ are pluripolar. Put
$\Sigma''_j:=\Sigma'_j\cup\bigcup_{k=1}^\infty C_{j,k}$.
Then $\Sigma''_j$ is pluripolar, $j=1,\dots,N$. Let $\BT'':=\BT((D_j,A_j,\Sigma''_j)_{j=1}^N)$.
Observe that $\BT''\cap\wdht M_0=\varnothing$, which contradicts Lemma \ref{LemQT}.

Thus there exists a pair $(j,k)$ such that $C_{j,k}$ is not pluripolar. We may assume that $j=N$.
Put $a:=a_k$, $\rho:=\rho_{a_k}$, $g:=g_{a_k}$.
Notice that for every $b'_N\in C_{N,k}$ there exists a $b_N\in\PP(a_N,\rho)$ such that
$g(b'_N,b_N)=0$. Put
$$
V:=\{z'_N\in\PP(a'_N,\rho): g(z'_N,\cdot)\equiv0 \text{ on } \PP(a_N,\rho)\}.
$$
Then $V$ is a proper analytic set and, therefore, the set
$C_{N,k}\setminus V$ is not pluripolar.

In the case where $n_N=1$ it suffices to take an arbitrary
$b'_N\in C_{N,k}\setminus V$ and we are in the situation of (*).

If $n_N\geq2$, then take an arbitrary $b'_N\in C_{N,k}\setminus V$ and a $b_N\in\PP(a_N,\rho)$
such that $g(b)=0$ with $b:=(b'_N,b_N)$. Since $g(b'_N,\cdot)\not\equiv0$, there exist a unitary
isomorphism $\BU:\CC^{n_N}\too\CC^{n_N}$ and $r>0$ such that $\PP(b,r)\subset\PP(a,\rho)$ and
for each $\overset{\approx}\xi\in\PP(0,r)\subset\CC^{n_N-1}$,
we have $g(b'_N,b_N+\BU(\overset{\approx}\xi,\cdot))\not\equiv0$ near zero.
Define
$$
\wdtl g(z):=g(z'_N,b_N+\BU(z_N-b_N)),\quad z=(z'_N,z_N)\in\PP(b,r).
$$
Then $\wdtl g(b)=0$ and $\wdtl g(b'_N,b_{N,1},\dots,b_{N,n_N-1},\cdot)\not\equiv0$.
Moreover,
$$
\wdtl g^{-1}(0)\cap((A'_N\setminus\Sigma_N)\times\PP(b,r))\subset\wdtl h^{-1}(0),
$$
where $\wdtl h(z):=g(z'_N,b_N+\BU(z_N-b_N))$, $z=(z'_N,z_N)\in\PP(b,r)$.
Thus, the new objects satisfy (*). Consequently, repeating the procedure in (*),
we conclude that $b\in\wdtl\Omega\cap\wdtl g^{-1}(0)\subset\wdtl h^{-1}(0)$
for an open neighborhood $\wdtl\Omega$ of $b$, which means that
$b\in\Omega\cap g^{-1}(0)\subset h^{-1}(0)$ for an open neighborhood $\Omega$ of $b$.
\end{proof}

\def\bibname{References}
\bibliographystyle{amsplain}

\end{document}